\documentclass[11pt,reqno]{amsart}

\usepackage{graphicx}
\usepackage[draft]{hyperref}
\usepackage{amsmath,amsopn,amssymb,amsfonts,stmaryrd}
\usepackage{verbatim}
\usepackage{amsthm}
\usepackage{mathtools}
\usepackage{color}
\usepackage{enumitem}
\usepackage{bbm}
\usepackage{mathrsfs}
\usepackage{booktabs}
\usepackage{caption}
\usepackage{bm}
\usepackage{tensor}
\usepackage{xcolor}
\usepackage{bbm}
\usepackage{cleveref}
\usepackage{enumitem}

\makeatletter
\@namedef{subjclassname@2020}{%
	\textup{2020} Mathematics Subject Classification}
\makeatother

\title[Appell--Lerch functions and false theta functions as $q$-brackets]{Quasi-Jacobi forms, Appell--Lerch functions, and\\ false theta functions as $q$-brackets\\ of functions on partitions}
\author{Kathrin Bringmann}  
\author{Jan-Willem van Ittersum} 
\author{Jonas Kaszian}
\date{} 

\usepackage{thmtools}
\usepackage{thm-restate}

\newtheorem{thm}{Theorem}
\newtheorem{lem}[thm]{Lemma}
\newtheorem{cor}[thm]{Corollary}
\newtheorem{prp}[thm]{Proposition}

\theoremstyle{definition}
\newtheorem*{definition}{Definition}

\newtheorem*{rmk}{Remark}

\theoremstyle{remark}

\numberwithin{thm}{section}
\numberwithin{prp}{section}
\numberwithin{lem}{section}
\numberwithin{cor}{section}
\numberwithin{equation}{section}
\numberwithin{conjecture}{section}

\crefname{prp}{proposition}{propositions}
\crefname{equation}{}{}

\newcommand{\Z}{\mathbb{Z}}
\newcommand{\Q}{\mathbb{Q}}
\newcommand{\N}{\mathbb{N}}
\newcommand{\R}{\mathbb{R}}
\newcommand{\C}{\mathbb{C}}

\renewcommand{\H}{\mathbb{H}}

\renewcommand{\t}{\tau}

\newcommand{\z}{\zeta}

\newcommand{\Pmod}[1]{\,\,({\rm mod}\,\,{#1})}

\newcommand{\rb}[1]{\left({#1}\right)}

\newcommand{\sgn}{\operatorname{sgn}}

\newcommand{\SL}{\operatorname{SL}}

\allowdisplaybreaks

\newcommand{\bdd}{\begin{center}\begin{tikzcd}}
\newcommand{\bd}{\begin{tikzcd}}
\newcommand{\edd}{\end{tikzcd}\end{center}}
\newcommand{\ed}{\end{tikzcd}}
\newcommand{\bdp}{\begin{center}\begin{tikzpicture}}
\newcommand{\edp}{\end{tikzpicture}\end{center}}
\newcommand{\bi}{\begin{itemize}}
\newcommand{\ei}{\end{itemize}}
\newcommand{\bt}{\begin{tikzpicture}}
\newcommand{\et}{\end{tikzpicture}}
\newcommand{\ba}{\[\begin{aligned}}
\newcommand{\ea}{\end{aligned}\]}
\newcommand{\bp}{\begin{pmatrix}}
\newcommand{\ep}{\end{pmatrix}}
\newcommand{\bsm}{\begin{smallmatrix}}
\newcommand{\esm}{\end{smallmatrix}}
\newcommand{\bv}{\begin{vmatrix}}
\newcommand{\ev}{\end{vmatrix}}
\newcommand{\bb}{\begin{bmatrix}}
\newcommand{\eb}{\end{bmatrix}}
\newcommand{\bB}{\begin{Bmatrix}}
\newcommand{\eB}{\end{Bmatrix}}
\newcommand{\bea}{\begin{enumerate}[leftmargin=*,label=\textnormal{(\alph*)}]}
\newcommand{\ber}{\begin{enumerate}[leftmargin=*,label=\textnormal{(\roman*)}]}
\newcommand{\ben}{\begin{enumerate}[leftmargin=*,label=\textnormal{(\arabic*)}]}
\newcommand{\ee}{\end{enumerate}}

\renewcommand{\vec}{\bm}
\newcommand{\partitions}{\mathscr{P}}
\newcommand{\QtoP}{\Q^{\partitions}}

\newcommand{\sltwoz}{\mathrm{SL}_2(\Z)}
\renewcommand{\=}{\: =\: }

\newcommand{\defis}{\: :=\: }
\newcommand{\+}{\,+\,}
\newcommand{\meno}{\,-\,}

\newcommand{\At}{A}
\newcommand{\Bt}{B}

\newcommand{\As}{\mathcal{A}}
\newcommand{\Bs}{\mathcal{B}}
\newcommand{\ps}{l}

\newcommand{\Qu}{\Q\llbracket\vec{u}\rrbracket}
\newcommand{\Quzeta}{\Q\left\llbracket\vec{u},\zeta,\zeta^{-1}\right\rrbracket}
\newcommand{\wt}{\mathrm{wt}\,}
\newcommand{\mob}{\mu(\alpha,\one)}
\newcommand{\one}{\mathbbm{1}_{\ps}}
\newcommand{\bigconv}{\mathop{\hspace{-0.0pt}\scalebox{3}{\raisebox{-0.4ex}{$\ast$}}}}
\newcommand{\pdv}[2]{\frac{\partial #1}{\partial #2}}

\newcommand{\abcd}{\left(\begin{smallmatrix} a & b \\ c & d \end{smallmatrix}\right)}

\newcommand{\coef}[1]{\operatorname{coeff}_{\left[#1\right]}}

\definecolor{darkgreen}{rgb}{0.0, 0.4, 0.26}

\usepackage{xr}

\makeatletter
\renewcommand{\boxed}[1]{\text{\fboxsep=.2em\fbox{\m@th$\displaystyle#1$}}}
\makeatother

\makeatletter
\@namedef{subjclassname@2020}{%
	\textup{2020} Mathematics Subject Classification}
\makeatother

\address{Department of Mathematics and Computer Science\\Division of Mathematics\\University of Cologne\\ Weyertal 86-90 \\ 50931 Cologne \\Germany}
\email{kbringma@math.uni-koeln.de}

\email{j.w.ittersum@uni-koeln.de}

\email{jonaskaszian@gmail.com}

\begin{document}
	\begin{abstract}
		We study certain algebras of theta-like functions on partitions, for which the corresponding generating functions give rise to theta functions, quasi-Jacobi forms, Appell--Lerch sums, and false theta functions.
	\end{abstract}
\maketitle

\section{Introduction and statement of results}
Let $\partitions$ be the set of all partitions of integers and for $\lambda \in \partitions$ denote by $|\lambda|$ the size of the partition~$\lambda$. For $f:\partitions\to \Q$, the \emph{$q$-bracket} of~$f$ is given by
\begin{equation}\label{eq:qbracket}
	\langle f \rangle_q \defis \frac{\sum_{\lambda \in \partitions} f(\lambda) q^{|\lambda|}}{\sum_{\lambda \in \partitions} q^{|\lambda|}} \in  \Q\llbracket q\rrbracket.
\end{equation}

For many classes of functions, the $q$-bracket is a quasimodular form 
(see, e.g. \cite{BO00, CMZ16, vISymmetric, vITaylor, Zag16}). The connection to modular forms is not too surprising, since the denominator of the $q$-bracket equals $q^{\frac{1}{24}}\eta(\t)^{-1}$ with $\eta(\t):= q^{\frac{1}{24}} \prod_{n\geq 1} (1-q^n)$ the {\it Dedekind eta function} (with $q:=e^{2\pi i \t}$). Notably, in several of these cases these ``classes of functions'' are algebras, which is surprising since the $q$-bracket is linear, but not an algebra homomorphism. This fact is crucial in many applications in enumerative geometry \cite{CMSZ19, CMZ16, Eng17, EO01, HIL19, Och01, IS22}. Interpreting the $q$-bracket as the expectation value of some partition invariant~$f$ with respect to the normalized probability distribution $\nu\mapsto \frac{q^{|\nu|}}{\sum_{\lambda\in \partitions} q^{|\lambda|}}$, by this property all the moments of~$f$ are modular objects. In short, the theory of $q$-brackets forms a bridge between enumerative geometry, statistical properties of partitions and number theory. 

To be more specific, two examples of such algebras are the algebra of so-called symmetric functions $\mathcal{S}$ and the algebra of shifted symmetric functions $\Lambda^*$. These algebras are generated by the \emph{symmetric functions}~$S_k$ and \emph{shifted symmetric functions}~$Q_k$, respectively, for $\lambda \in \partitions$ given by ($\lambda_j$ is the $j$-th part of $\lambda$ and $B_k$ the $k$-th Bernoulli number)
\begin{align*}
	S_k(\lambda) &\defis -\frac{B_k}{2k}\+\sum_{j\ge1}  \lambda_j^{k-1} \qquad \text{($k\geq 2$ even)}\\
		Q_k(\lambda) &\defis -\left(1-2^{1-k}\right)\frac{B_k}{k} \+ \sum_{j\ge1} \left(\left(\lambda_j-j+\frac{1}{2}\right)^{k-1}-\left(-j+\frac{1}{2}\right)^{k-1}\right) \qquad (k\geq 2),
\end{align*}
i.e., $\mathcal{S}=\Q[S_2,S_4,S_6,\ldots]$ and $\Lambda^*=\Q[Q_2,Q_3,Q_4,\ldots]$. These algebras are graded by $\wt(S_k)=k$ and $\wt(Q_k)=k$. For all homogeneous $f\in \mathcal{S}$ or $f\in \Lambda^*$ the $q$-bracket~$\langle f \rangle_q$ is a quasimodular form \cite{BO00, vISymmetric}. In fact, every function $f\in \QtoP$ such that all moments~$\langle f^n\rangle_q$ are quasimodular for $\sltwoz$ seems to be an element of either of these two algebras.

Motivated by the study of invariants of flat surfaces, in \cite{CMZ16}, extensions of the above results for $\Lambda^*$ were found, where the $q$-brackets are Eichler integrals of modular forms (see also \cite{BOW20}). In this paper, we aim to get a better understanding of further modularity properties such $q$-brackets may admit. In particular, we introduce an algebra $\mathcal{A}$ of theta-like functions on partitions and explain how several types of modularity arise if such functions are paired with elements of $\mathcal{S}$ and their generalizations. More concretely, let, for $N\in\N$,
\begin{equation}\label{eq:fa} 
	t_{(N)}(\zeta;\lambda) 
	\defis 1 \+ \sum_{m\ge1} (\zeta^m+\zeta^{-m})\,\delta_{r_m(\lambda) \geq Nm}.
\end{equation}
Here, $\delta_X:=1$ if a statement $X$ is true and $0$ otherwise, and $r_m(\lambda)$ denotes the multiplicity of parts of size~$m$ in~$\lambda$. The probability distribution of~$r_m(\lambda)$ was determined in~\cite{Fri93}. In particular, as $|\lambda|\to \infty$,
\[ 
	P\!\left(r_m(\lambda)\geq Nm\right) \sim 
	\begin{cases}
		1 & \text{if }m=o\!\left(|\lambda|^\frac14\right),\\
		0 & \text{if }|\lambda|=o\!\left(m^4\right).
	\end{cases}
\]
We let $\mathcal{A}$ be the algebra generated by the functions $\lambda\mapsto t_{(N)}(\zeta;\lambda)$ for $N\in \N$ and $[\As]_{\zeta_\ell=1}$ the algebra generated by the functions 
\[\lambda \mapsto t_{(N)}(1;\lambda) = 1 \+ 2\sum_{m\ge1}\delta_{r_m(\lambda) \geq Nm}. \]
Our first theorem states that $\As$ gives rise to an algebra of theta-like functions on partitions.
\begin{thm}\label{thm:1}
	The vector space $\As$ satisfies the following properties:
\begin{enumerate}[leftmargin=*,label=\normalfont(\arabic*)]
	\item[\upshape{(1)}] If $f \in[\As]_{\zeta_\ell=1}$, then $\langle f \rangle_q$ is a modular theta function of some level. 
	\item[\upshape{(2)}] If $f \in\As$, then $\langle f \rangle_q$ is a Jacobi theta function.
	\end{enumerate}
\end{thm}
\begin{rmk}\itshape
In particular, by {\rm(1)}, for $f\in[\Phi(\As)]_{\zeta_\ell=1}$, the moments $\langle f^n\rangle_q$ are modular forms for all $n\in \N$. This contrasts the situation for $\Lambda^*$ and $\mathcal{S}$ for which the vector subspace of elements with a modular $q$-bracket is not a subalgebra (see~\cite{vIWhen}). 
 \end{rmk}

Moreover, we consider the algebra $\Bs$, generated by
\begin{equation*}
	s(\varrho;\lambda) := \frac{\varrho}{1-\varrho}+\frac{1}{2}\+ \sum_{m\ge1} \left(  \varrho^m \meno \varrho^{-m} \right)r_m(\lambda).
\end{equation*}
 %
Note that $s$ is a generating function for the $S_k$, i.e., 
\begin{equation}\label{eq:sgenser} 
	s(e^z;\lambda) \= -\frac{1}{z}+2\sum_{k\geq 2 \text{ even}} S_{k}(\lambda) \frac{z^{k-1}}{(k-1)!}.
\end{equation}

The elements of $\Bs$ give rise to Jacobi forms. Moreover, also products of elements of $\As$ and $\Bs$ give rise to Jacobi forms. 
\begin{thm}\label{thm:2} 
	If $f\in \As$ and $g\in \Bs$, then $\langle fg\rangle_q$ is a meromorphic quasi-Jacobi form.
\end{thm}
\begin{rmk}\itshape
For comparison, if $f\in \Lambda^*$, $g\in \mathcal{S}$, then $\langle f\rangle_q$ and $\langle g\rangle_q$ are quasimodular, but $\langle f g\rangle_q$ is not necessarily a quasimodular form. The latter $q$-bracket lies in a space of multiple Eisenstein series, also called the space of $q$-analogues of multiple zeta values (see~\cite{BI22}).
\end{rmk}
  \begin{cor} 
Let $f$ be in the algebra generated by the $S_k$ for $k\geq 0$ even and the functions 
 \[
\lambda\mapsto 1 \+ 2\sum_{m\ge1}m^\ell\delta_{r_m(\lambda) \geq Nm}.
\]
for $N\in \N$ and $\ell\geq 0$ even. Then the $q$-bracket $\langle f\rangle_q$ is a quasimodular form.
 \end{cor}
Finally, consider 
\begin{equation*}
	s^*(\varrho,\xi;\lambda) = \frac{1}{1-\varrho}\+ \frac{\xi}{1-\xi}\+\sum_{m\ge1}\sum_{r=1}^{r_m(\lambda)} (\varrho^m\xi^r \meno \varrho^{-m}\xi^{-r}).
\end{equation*}
Similarly, taking the coefficients of~$s^*$ around $w=\mathfrak{z}=0$ with $\varrho=e^{2\pi i \mathfrak{z}}$ and $\xi=e^{2\pi i w}$ generates an algebra~$\mathcal{T}$ containing $\mathcal{S}$, but which is not a graded algebra.  
We refer the interested reader to Subsection \ref{sec:algebras} or~\cite{vISymmetric} for more details about the algebras $\mathcal{S}$ and $\mathcal{T}$.

Quasi-Jacobi forms admit non-holomorphic completions which transform as Jacobi forms. We note that by taking $q$-brackets with $s^*$ yields forms for which the completion is more involved than for quasi-Jacobi forms; see \Cref{sec:falsemock} for a more precise statement.
\begin{thm}\label{thm:falsemock}
For $N\in\N$, the bracket~$\langle t_{(N)} s^*\rangle_q$ can be expressed as a linear combination of a Jacobi theta function, a false theta function, and an Appell--Lerch sum.
\end{thm}
In \Cref{sec:pre} we recall the relevant background, such as the connected $q$-bracket, quasi-Jacobi forms, Appell functions, and false theta functions. The proof of the first two theorems can be found in~\Cref{sec:12} and the proof of \Cref{thm:falsemock} in \Cref{sec:falsemock}. 

\section*{Acknowledgements}
The authors were supported by the SFB/TRR 191 “Symplectic Structure in Geometry, Algebra and Dynamics”, funded by the DFG (Projektnummer 281071066 TRR 191). Part of this paper was done while the last two authors were at the Max Planck Institute for Mathematics. The authors thank Walter Bridges for helpful conversations.

\section{Preliminaries}\label{sec:pre}
\subsection{Partitions and the connected $q$-bracket}
Each partition $\lambda \in \partitions$ is represented as $\lambda= (\lambda_1,\lambda_2,\ldots)$ with ${\lambda_j\in \N_0}$, $\lambda_1\geq \lambda_2\geq \cdots$ and $\lambda_j=0$ for all but finitely many $j$, or, equivalently, as a finite multiset where the part $m\in \lambda$ has multiplicity $r_m(\lambda)$. We write $|\lambda|=\sum_j\lambda_j = \sum_{m\geq 1} m r_m(\lambda)$ for the size of $\lambda$ and $\ell(\lambda):=|\{j: \lambda_j\neq 0\}|=\sum_{m\geq 1} r_m(\lambda)$ for the length of~$\lambda$. 
Following \cite{vISymmetric}, we refine the $q$-bracket (see \eqref{eq:qbracket}) by the \emph{$\vec{u}$-bracket}, which is the isomorphism of vector spaces $\QtoP \to \Qu$, where $\vec{u}=(u_1,u_2,\ldots)$, given by
\begin{align}
f\mapsto 	\langle f \rangle_{\vec{u}} \defis \frac{\sum_{\lambda \in \partitions} f(\lambda) u_\lambda}{\sum_{\lambda \in \partitions} u_\lambda} \qquad (u_\lambda = u_{\lambda_1} u_{\lambda_2}\cdots ).
\end{align}
Note that for $f\in \QtoP$, one has $\langle f \rangle_q = [\langle f \rangle_{\vec{u}}]_{u_j=q^j}\mspace{1mu}$.
A partition is {\it strict} if it has no repeated parts, i.e., if $r_m(\lambda)\leq 1$ for all $m\in \N$. Define the {\it M\"obius function on partitions} by
\begin{equation}\label{eq:moebius}
	\mu(\lambda) := 
	\begin{cases} 
		(-1)^{\ell(\lambda)} & \text{if $\lambda$ is strict},\\
		0 & \text{otherwise}. 
	\end{cases}
\end{equation}
Then, we have
\[ \langle f\rangle_{\vec{u}} = \sum_{\lambda \in \partitions} f(\lambda) u_\lambda\sum_{\lambda \in \partitions} \mu(\lambda) u_\lambda.\]

If $f(\lambda)=g(r_m(\lambda))$ for some $m$, then \cite[Proposition~3.1.4]{vISymmetric} explains how to compute its $\vec{u}$-bracket. Let the \emph{discrete derivative}~$\partial$ of $g:\N\to \Q$ be defined by $\partial g(n) := g(n)-g(n-1)$ for $n\geq 2$ and $\partial g(1):=g(1)$. 
\begin{lem}\label{lem:314}
If $f(\lambda)=g(r_m(\lambda))$ with $m\in \N$, then
\[ \langle f\rangle_{\vec{u}} = \sum_{r\ge1} \partial g(r) u_m^r.\]
\end{lem}

We define the connected~$q$-bracket \cite[p.~55--57]{CMZ16}, which naturally arises in enumerative geometry from counting connected coverings. From a combinatorial point of view, often the connected~$q$-bracket is easier to compute than the usual~$q$-bracket.
For a set $X$, write $\pi(X)$ for the set of all set partitions of $X$. For $\ps\in\N$, let $[\ps]:=\{1,\ldots,\ps\}$. Write $\Pi(\ps)$ for the set of all set partitions of $[\ps]$. For example, we write $\one=\{[\ps]\}$ for the partition consisting of one element. For $A\subset [\ps]$ we denote $f_A = \prod_{a\in A} f_a$. The \emph{connected~$q$-bracket} is defined as the multilinear map $\langle\cdot\rangle_q:(\QtoP)^{\otimes \ps}\to \Q$ extending the~$q$-bracket such that for $f_1,\ldots, f_\ps\in \QtoP$
\begin{align*} 
	\langle f_1 \otimes \cdots \otimes f_\ps\rangle_q \= \sum_{\alpha \in \Pi(\ps)} \mu(\alpha,\one)\prod_{A\in \alpha}\langle f_A\rangle_q. 
\end{align*}
Here, for $\alpha,\beta\in \Pi(\ps)$ the \emph{M\"obius function} is given by
\begin{equation*}
\mu(\alpha,\beta):\=\prod_{B\in\beta}(-1)^{|\alpha_B|+1}(|\alpha_B|-1)!,
\end{equation*}
 where~$\alpha_B$ for $B\subset [\ps]$ is the partition on~$B$ induced by~$\alpha$. We write $\alpha\leq \beta$ if for all~$A\in \alpha$ there exists a~$B\in \beta$ such that~$A\subseteq B$.
By invoking the M\"obius inversion formula one finds
\begin{align}\label{eq:connectedbracket}
	\prod_{B\in \beta}\langle \otimes_{b\in B}f_b \rangle_q &\= \sum_{\alpha\leq\beta} \mu(\alpha,\beta)\prod_{A\in \alpha}\langle f_A\rangle_q , \qquad
	\prod_{A\in \alpha}\langle f_A\rangle_q\= \sum_{\beta\leq\alpha} \prod_{B\in \beta}\langle \otimes_{b\in B}f_b \rangle_q .
\end{align}
Similarly, we let
\begin{align} \label{eq:conubrack}
	\langle f_1 \otimes \cdots \otimes f_\ps\rangle_{\vec{u}} :\= \sum_{\alpha \in \Pi(\ps)} \mu(\alpha,\one)\prod_{A\in \alpha}\langle f_A\rangle_{\vec{u}}
\end{align}
and have a M\"obius inversion formula
\begin{align}\label{eq:mobiusubrack}
	\prod_{A\in \alpha}\langle f_A\rangle_{\vec{u}}\= \sum_{\beta\leq\alpha} \prod_{B\in \beta}\langle \otimes_{b\in B}f_b \rangle_{\vec{u}} .
\end{align}

\subsection{The algebras $\mathcal{S}, \mathcal{T}$, and $\Lambda^*$}\label{sec:algebras}\hspace{0cm}
In the introduction, the algebras $\Lambda^*$ and $\mathcal{S}$ are defined and it is explained that the $q$-brackets of their elements are quasimodular. Besides this property which these algebras share, there are many differences. For example, for $f\in \Lambda^*$ of weight $k$, one has $f(\lambda')=(-1)^k f(\lambda)$, where $\lambda'$ denotes the conjugate of the partition $\lambda$, whereas there is no easy relation by $f(\lambda')$ and $f(\lambda)$ for $f\in \mathcal{S}$. 

The Taylor coefficients of $s^*$ give rise to an extension of $\mathcal{S}$. More concretely, let
\[
T_{k,\ell}(\lambda) \defis -\frac{B_{k+\ell}}{2(k+\ell)}\delta_{k=0\,\text{or}\,\ell=1}\+\sum_{m\ge1}\sum_{r=1}^{r_m(\lambda)} m^{k} r_m(\lambda)^{\ell-1} \qquad (k\geq 0, \ell\geq 1, k+\ell \text{ even}).
\]
Then, the $T_{k,\ell}$ are the Taylor coefficients of $s^*$. In particular, note that $T_{k-1,1}=S_{k}$. The algebra~$\mathcal{T}$ generated by the $T_{k,\ell}$ also has the property that the $q$-bracket of $f\in \mathcal{T}$ is quasimodular. Moreover, it has the interesting property that $\langle \mathcal{T}\rangle_{\vec{u}}$ is also an algebra with the multiplication of power series in $\Qu$. In contrast, $\langle \Lambda^*\rangle_{\vec{u}}$ and $\langle \mathcal{S}\rangle_{\vec{u}}$ are not algebras. 
Further variations of these algebras are known for which the elements are quasimodular under a modification of the $q$-bracket \cite{Eng17, EO06, EOP08}, e.g., for the measure $\lambda \mapsto \frac{\mu(\nu) q^{|\nu|}}{\sum_{\lambda \in \partitions} \mu(\lambda) q^{|\lambda|}}$.

\subsection{Jacobi forms} Define $\zeta=e^{2\pi i z}$, $q=e^{2\pi i \t}$ and for a complex variable $w$, write $w=w_1+iw_2$ for $w_1,w_2\in \R$ throughout. The \emph{Jacobi theta function} is given by
\begin{equation}\label{eq:Jacobitheta} 
	\vartheta(z;\tau):=\sum_{n\in \Z+\frac12} e^{2\pi in\left(z+\frac12\right)} q^{\frac{n^2}2} .
\end{equation}
The Jacobi theta function is a Jacobi form of weight~$\frac{1}{2}$ and index~$\frac{1}{2}$, that is, for $\gamma =\abcd\in\SL_2(\Z)$ and $m,\ell\in \Z$, we have (which follows from Table~V of~\cite{Mum83})
\begin{align}
\vartheta(z+m\tau+\ell;\tau) &=(-1)^{m+\ell}\zeta^{-m}q^{-\frac{m^2}{2}}\vartheta(z;\tau),\nonumber\\
\vartheta\!\left(\frac{z}{c\tau+d};\frac{a\tau+b}{c\tau+d}\right) &= \nu_\eta^3(\gamma)\sqrt{c\tau+d} e^{\frac{\pi icz^2}{c\t+d}}\vartheta(z;\tau),\label{eq:theta}
\end{align}
where $\nu_\eta$ is the multiplier of $\eta$, that is, for $\gamma =\abcd \in \SL_2(\Z)$, we have
\[
\eta\left(\frac{a\tau+b}{c\tau+d}\right) = \nu_\eta(\gamma) \sqrt{c\tau+d}\; \eta(\tau).
\]
Explicitly $\nu_\eta$ is given by (see~\cite{Kn})
\begin{equation}\label{E:MultEt}
	\nu_\eta
	\begin{pmatrix}
		a & b\\
		c & d
	\end{pmatrix}
	=
	\begin{cases}
		\left(\frac d{|c|}\right)e^{\frac{\pi i}{12}\left((a+d)c-bd\left(c^2-1\right)-3c\right)} & \text{if $c$ is odd},\\
		\left(\frac cd\right)e^{\frac{\pi i}{12}\left(ac\left(1-d^2\right)+d(b-c+3)-3\right)} & \text{if $c$ is even},
	\end{cases}
\end{equation}
where $(\frac{\cdot}{\cdot})$ denotes the Kronecker symbol.
Write
\[ \Theta(z;\tau):=\sum_{n\in \Z} \zeta^n q^\frac{n^2}2
= \zeta^{\frac{1}{2}}q^{\frac{1}{8}}\sum_{n\in\Z+\frac{1}{2}} \zeta^{n} q^{\frac{n^2+n}{2}} =  \zeta^{\frac{1}{2}}q^{\frac{1}{8}} \vartheta\left(z-\frac12+\frac\t2;\t\right).
\]
Then, for $\gamma =\abcd\in\Gamma(2)$ and $m,\ell\in \Z$, we have \cite[p.~2 and Theorem 7.1]{Mum83}
\begin{align*}
\Theta(z+m\tau+\ell;\tau) &=\zeta^{-m}q^{-m^2}\Theta(z;\tau),\\
\Theta\!\left(\frac{z}{c\tau+d};\frac{a\tau+b}{c\tau+d}\right) &= \left(\frac{-2c}{d}\right) \varepsilon_d^{-1} \sqrt{c\tau+d}\; e^{\frac{\pi icz^2}{c\t+d}}\Theta(z;\tau),
\end{align*}
where, for $d$ odd, we let $\varepsilon_d:=1$ if $d\equiv1\Pmod4$ and $\varepsilon_d:=i$ if $d\equiv3\Pmod4$.

Another Jacobi form that we encounter in this paper is the {\it Weierstrass $\wp$-function} 
\[
	\wp(z;\tau) := \frac{1}{z^2} + \sum_{\substack{\omega \in \Z\tau+\Z \\ \omega\neq 0}} \left(\frac{1}{(z-\omega)^2}-\frac{1}{\omega^2}\right).
\]
We have for $m$, $n\in\Z$ and $
\begin{psmallmatrix}
	a & b\\
	c & d
\end{psmallmatrix}
\in\SL_2(\Z)$ \cite{EZ85}
\[ \wp(z+m\tau+n\tau;\tau)=\wp(z;\tau), \qquad \wp\!\left(\frac{z}{c\tau+d};\frac{a\tau+b}{c\tau +d}\right) = (c\tau+d)^{-2} e^{-\frac{2\pi icz^2}{c\t+d}}\wp(z;\tau).\]

For $\gamma=\abcd$ ($\in\SL_2(\Z)$ if $k\in\Z$ and $\in\Gamma_0(4)$ if $k\in\Z+\frac12$) we define \emph{the automorphy factor}~$j(\gamma,\tau)$ by
\begin{equation}\label{eq:aut}
	j(\gamma,\tau) :=
	\begin{cases}
		\sqrt{c\tau+d} & \text{if $k\in\Z$},\\
		\left(\frac cd\right) \varepsilon_d^{-2k} \sqrt{c\tau+d} & \text{if $k\in\Z+\frac12$}.
	\end{cases}
\end{equation}
The functions obtained as $q$-brackets in this paper are Jacobi forms, defined below, possibly after adding non-holomorphic parts as we discuss in the sections below. 
\begin{definition}{~}
Let $k$, $M\in\frac12\Z$. Let $\Gamma$ be a congruence subgroup of $\sltwoz$ if $k\in\Z$, and of $\Gamma_0(4)$ if $k\in \Z +\frac{1}{2}$. A \emph{polar Jacobi form} of weight $k$, index $M$, and character $\chi$ on $\Gamma$ is a function $\phi:\C\times \H\to \C$ that is holomorphic in $\tau$, meromorphic in $z$, and that satisfies:
\begin{enumerate}[leftmargin=*]
\item For $m,\ell\in \Z$ we have that
\[ \phi(z+m\tau+\ell;\tau) = (-1)^{M(\ell+m)} e^{-2\pi iM\left(m^2\tau+2m z\right)}\phi(z;\tau).\]
\item For $\gamma=\abcd\in \Gamma$ we have that
\[ \phi\!\left(\frac{z}{c\tau+d};\frac{a\tau+b}{c\tau+d}\right) = \chi(d)j(\gamma,\tau)^{2k} e^{\frac{2\pi icMz^2}{c\tau+d}}\phi(z;\tau).\]
\item All poles $z_0$ of $z\mapsto \phi(z;\t)$ satisfy $z_0\in \Q\tau+\Q$. Moreover, if $z\in \Q\tau+\Q$ is not a pole of $\phi(z;\tau)$, then $\tau\mapsto \phi(z;\tau)$ is a weakly holomorphic modular form (of some level). 
\end{enumerate}
\end{definition}
Both $\vartheta$ and $\wp$ are polar Jacobi forms. It was discovered by Zwegers~\cite{Zw} that meromorphic Jacobi forms for which poles only occur at torsion points $z\in \Q\tau+\Q$ have a modified theta expansion related to mock modular forms. 
The notion of polar Jacobi forms is closely related to the so-called strictly meromorphic Jacobi forms in~\cite{vITaylor}.

We mention two other examples of Jacobi forms we encounter in this work. First of all, consider Zagier's polar Jacobi form (see Section 3 of \cite{Zag91}) $(q=e^{2\pi i \tau}, \zeta=e^{2\pi i z}, \xi=e^{2\pi i w})$ 
\begin{align}\label{eq:F} 
	F(z,w;\tau):=& \sum_{n\ge0} \frac{\xi^{-n}}{q^{-n}\zeta-1}-\sum_{m\ge0} \frac{\zeta^m\xi}{q^{-m}-\xi} \\
	=&\frac{\zeta\xi-1}{(\zeta-1)(\xi-1)}-\sum_{m,r\geq 1} \left(\z^m\xi^r-\zeta^{-m}\xi^{-r}\right)q^{mr} \nonumber
\end{align}
defined for $z_2<\tau_2$ and $-w_2<\tau_2$. 
The following identity was known to Kronecker \cite{Kro}, \cite[pp.~309--318]{Kro2}; see also \cite[pp. 70-71]{Wei76}, \cite[Section~3]{Zag91}, \cite[equation ~(1.1)]{Mor17}
\[ 
	F(z,w;\tau) = \frac{1}{2\pi i}\frac{\vartheta'(0;\tau)\vartheta(z+w;\tau)}{\vartheta(z;\tau)\vartheta(w;\tau)},
\]
where $\vartheta'(z;\tau) := \pdv{}{z}\vartheta(z;\tau)$. Moreover, by \cite[Theorem~4.4.1]{vISymmetric} we have 
\begin{equation}\label{eq:shatbrac}\langle s^*(\z,\xi) \rangle_q = F(z,w;\tau).\end{equation}

\begin{rmk}\itshape
For $n\in\N$, define\footnote{Note that, in this remark, $z_1$ and $z_2$ should not be confused with the real and the imaginary part of $z$.}
\begin{align*}
	&F_n(z_1,\ldots,z_{2n};\t)\\ 
	&:=  (-2)^{-n}\!\sum_{\vec{m},\vec{r}\in \Z^n}\! \z_{2n+1}^{m_{n}}\z_{2n+2}^{r_n} q^{\vec{m}\cdot\vec{r}}\prod_{j=1}^n \left(\!\sgn\left(m_j-m_{j+1}-\frac{1}{2}\right)+\sgn\left(r_1+\ldots+r_j+\frac{1}{2}\right)\!\right),
\end{align*}
where we set $m_{n+1}=0$. This function is closely related to the generating function of $q$-analogues of multiple zeta values, given by 
\[
\sum_{\substack{m_1>\cdots>m_n>0 \\ r_1,\ldots,r_n>0}}\! \z_{2n+1}^{m_{n}}\z_{2n+2}^{r_n} q^{\vec{m}\cdot\vec{r}}.
\]
The appearance of $n$ sign factors seems to suggest that $F_n$ is not a Jacobi form. However, surprisingly, we have
\[F_n(z_1,\ldots,z_{2n};\t)= \prod_{j=1}^n F\!\left(z_1+z_3+\ldots+z_{2j+1},z_{2j}-z_{2j+2};\t\right).\]
\end{rmk}

Secondly, the \emph{Eisenstein--Kronecker series} $E_k$ are polar Jacobi forms of integer weight $k\geq 3$ and index $0$, given by (see \cite[Chapter~III~§2]{Wei76})
\[ E_k(z;\tau) := \sum_{c,d\in\Z}\frac{1}{(z+c\tau+d)^k}.\]

\subsection{Quasi-Jacobi forms}\label{sec:quasiJacforms} 
Extending the above definition of the Kronecker--Eisenstein series $E_k$ by the Eisenstein summation procedure
 $$
 	\sum_{c,d\in \Z} = \lim_{C\to\infty}\sum_{c=-C}^C\left(\lim_{D\to\infty}\sum_{d=-D}^D\right),
 $$
 the resulting function $E_1(z;\tau)$ and $E_2(z;\tau)$ do not satisfy the transformation properties of a Jacobi form. Note that $E_2(z;\tau)- e_2(\tau)= \wp(z;\tau)$, where $e_2(\tau):=\sum_{c,d\in \Z} \frac{1}{(c\tau+d)^2}$ with the summation is taken by the same Eisenstein summation procedure. Recall their completions (see equations~(4) and (6) in \cite[Chapter~VI~§2]{Wei76} with $u=1, v=\t, \delta=1$, and $2\pi iA=\tau-\overline{\tau}$)
\[ \widehat{E}_1(z;\tau) :=  E_1(z;\tau) + 2\pi i \frac{z_2}{\tau_2}, \qquad \widehat{E}_2(z;\tau) = E_2(z;\tau) -  \frac\pi{\tau_2},\]
which transform as Jacobi forms of weight $1$ and $2$ and index $0$, respectively. More generally, an \emph{almost Jacobi form} is a polynomial in $\frac{1}{\tau_2}$ and $\frac{z_2}{\tau_2}$ with meromorphic coefficients, which transforms as a Jacobi form:
\begin{definition} Let $k$, $M\in\frac12\Z$. Let $\Gamma$ be a congruence subgroup of $\sltwoz$ if $k\in\Z$, and of $\Gamma_0(4)$ if $k\in\Z+\frac12$. 
A \emph{polar almost Jacobi form} of weight $k$, index $M$, and character $\chi$ on $\Gamma$ is a function $\phi:\C\times \H\to \C$ satisfying:
\begin{enumerate}[leftmargin=*]
\item For $m,\ell\in \Z$ we have that
\[ \phi(z+m\tau+\ell;\tau) = (-1)^{M(\ell+m)} e^{-2\pi iM\left(m^2\tau+2m z\right)}\phi(z;\tau).\]
\item For $\gamma=\abcd\in \Gamma$ we have that
\[ \phi\!\left(\frac{z}{c\tau+d};\frac{a\tau+b}{c\tau+d}\right) = \chi(d) j(\gamma,\tau)^{2k} e^{\frac{2\pi icMz^2}{c\tau+d}}\phi(z;\tau).\]
\item We have
\[\phi(z;\tau)=\sum_{\ell,j} \phi_{\ell,j}(z;\tau) \frac{z_2^j}{\tau_2^{\ell+j}}\]
 with $\phi_{\ell,j}:\C\times \H\to \C$ meromorphic with all all poles $z_0$ of $z\mapsto \phi_{\ell,j}(z;\t)$ satisfy $z_0\in \Q\tau+\Q$. Moreover, if $z\in \Q\tau+\Q$ is not a pole of $\phi_{\ell,j}(z;\tau)$, then the function $\tau\mapsto \phi_{\ell,j}(z;\tau)$ is weakly holomorphic. 
\end{enumerate}
We say $\psi$ is a \emph{polar quasi-Jacobi form} if $\psi=\phi_{0,0}$ for some polar almost Jacobi form $\phi$. 
\end{definition}
As a first example, quasimodular forms such as $e_2$ can be seen as quasi-Jacobi forms which do not depend on the elliptic variable~$z$. Quasi-Jacobi forms naturally appear in enumerative geometry and the theory of vertex operator algebras; the definition and first study go back to these works \cite{Lib11, KM15, OP19}.

We recall a key properties of quasi-Jacobi forms. Write $J_{k,M}$ and $\widetilde{J}_{k,M}$ for the space of polar Jacobi forms and polar quasi-Jacobi forms, respectively, of weight~$k$ and index~$M$. 
\begin{prp}[{\cite[Proposition~2.18]{vITaylor}}] For $k$, $M\in\frac12\Z$
\[\widetilde{J}_{k,M} = J_{k,M}[E_1,E_2].\]
\end{prp}
Moreover, from \cite[Subsubsection~1.3.3]{OP19} and by \cite[Proposition~2.24]{vITaylor} one has
\begin{prp}\label{prop:der}
The maps $\frac{1}{2\pi i}\pdv{}{\tau}:\widetilde{J}_{k,M}\to \widetilde{J}_{k+2,M}$ and $\pdv{}{z}:\widetilde{J}_{k,M}\to \widetilde{J}_{k+1,M}$ are well-defined derivations.
\end{prp}
As the specialization $z=0$ of a quasi-Jacobi form yields a quasimodular form, this implies quasi-Jacobi forms admit a Laurent expansion in~$z$ for which the coefficients are quasimodular forms.

\subsection{Appell functions} 
Zwegers' {\it multivariable Appell functions} \cite{Zw2} are given by
	\begin{equation}\label{eq:appel}
		A_\ell(z,w;\t) := \z^\frac\ell2 \sum_{m\in\Z} \frac{(-1)^{\ell m} q^\frac{\ell m(m+1)}2 \xi^m}{1-\z q^m} \quad \left(\z=e^{2\pi iz}, \xi:=e^{2\pi iw}\right).
	\end{equation}
	Now recall the completions
	\begin{equation*}
		\widehat A_\ell(z,w;\t) = A_\ell(z,w;\t) + \frac i2 \sum_{k=0}^{\ell-1} \z^k \vartheta\left(w+k\t+\frac{\ell-1}2;\ell\t\right) R\!\left(\ell z-w-k\t-\frac{\ell-1}2;\ell\t\right),
	\end{equation*}
where $\vartheta$ is the Jacobi theta function defined in \eqref{eq:theta} and
\begin{align}\label{eq:R}
	R(z;\t) &:= \sum_{n\in\Z+\frac12} \left(\sgn(n)-E\!\left(\left(n+\frac{z_2}{\tau_2}\right)\sqrt{2\tau_2}\right)\right) (-1)^{n-\frac12} q^{-\frac{n^2}2} e^{-2\pi inz},
\end{align}
where for $x\in \R$ we define $E(x) := 2 \int_0^x e^{-\pi t^{2}}  dt$. By \cite[Theorem 2.2]{Zw}, for $m_1, m_2, \ell_1, \ell_2\in\Z$ we have the \emph{elliptic transformation}
	\begin{equation}\label{E:ellA}
		\widehat A_\ell(z+m_1\t+\ell_1,w+m_2\t+\ell_2;\t) = (-1)^{\ell(m_1+\ell_1)} \z^{\ell m_1-m_2} \xi^{-m_1} q^{\frac{\ell m_1^2}2 -m_1m_2} \widehat A_\ell(z,w;\t).
	\end{equation}
	For $\abcd\in \Gamma$, 	we have
	\begin{equation}\label{E:transA}
		\widehat A_\ell\left(\frac{z}{c\t+d},\frac{w}{c\t+d};\frac{a\t+b}{c\t+d}\right) = (c\t+d) e^{\frac{\pi ic}{c\t+d}\left(-\ell z^2+2zw\right)} \widehat A_\ell(z,w;\t).
	\end{equation}	
	Among other places, these Appell--Lerch sums occur in the decomposition of meromorphic Jacobi forms as in~\cite{DMZ14}; see also~\cite{BR10, HM14}.
	
\subsection{False theta functions}
In~\cite{BN}, false theta functions were embedded in a modular framework. For example, 
\begin{equation}\label{eq:false}
	\psi(z;\t) := \sum_{n\in\Z+\frac{1}{2}} \sgn(n)  e^{2\pi in\left(z+\frac12\right)} q^\frac{n^2}2,
\end{equation}
was obtained as the limit $\psi(z;\t)=\lim_{t\to \infty} \widehat{\psi}(z;\tau,\tau+it+\varepsilon)$ if $-\frac{1}{2}<\frac{z_2}{\tau_2}<\frac{1}{2}$ and $\varepsilon>0$ arbitrary, where
\begin{equation*}
	\widehat{\psi}(z;\t,\omega) := \sum_{n\in\Z+\frac{1}{2}} E\!\left(-i\sqrt{\pi i(\omega-\tau)}\left(n+\frac{z_2}{\tau_2}\right)\right) e^{2\pi in\left(z+\frac12\right)} q^{\frac{n^2}{2}}.
\end{equation*}
Note that
\begin{align}
	\widehat{\psi}(z+m\tau+\ell;\tau,\omega) &= (-1)^{m+\ell} \zeta^{-m} q^{-\frac{m^2}2} \widehat{\psi}(z;\tau,\omega),\label{el} \\
	\widehat{\psi}\left(\frac{z}{c\tau+d}; \frac{a\tau+b}{c\tau+d},\frac{a\omega+b}{c\omega+d}\right) &= \chi_{\tau,\omega}(\gamma) \nu_\eta(\gamma)^3 (c\tau+d)^{\frac{1}{2}} e^\frac{\pi icz^2}{c\tau+d} \widehat{\psi}(z;\tau,\omega),\label{mod}
\end{align}
for $\gamma=\left(\begin{smallmatrix} a & b \\ c & d\end{smallmatrix}\right)\in \sltwoz$ and $\ell,m\in \Z$.  Here,
\[
\chi_{\tau,\omega}(\gamma) := \sqrt{\frac{i(\omega-\tau)}{(c\tau+d)(c\omega+d)}}\frac{\sqrt{c\t+d}\sqrt{c\omega+d}}{\sqrt{i(\omega-\t)}}.
\]

\section{Proof of \cref{thm:1} and \cref{thm:2}}\label{sec:12}
Set $u_0:=1$
and  for $m\in\N$ let $u_{-m}\defis u_m^{-1}$. We consider the ``theta function''
\[ \At_{N}(\zeta,\vec{u}) 
:= \sum_{m\in \Z} \zeta^mu_m^{Nm}\in \Quzeta \qquad (N \in \N),\]
which under the substitution $u_m=q^m$ equals an actual theta function. Moreover, we consider the ``theta function associated to a hyperbolic form''
\begin{align}\label{eq:hyptheta}
	\frac{1}{2}\sum_{m,r\in \Z} \left(\sgn\left(m+\frac{1}{2}\right)+\sgn\left(r-\frac{1}{2}\right)\right) \varrho^m \xi^r u_m^r,
\end{align}
for which the constant term around $w=0$, where $\xi=e^{2\pi i w}$, equals 
\begin{align*}
	\Bt(\varrho,\vec{u}) :&= 
 \frac{\varrho}{1-\varrho}+\frac{1}{2} + \sum_{\substack{m,r\in \Z\\ mr>0}} \sgn(m) \varrho^m u_m^r.
\end{align*}
From \eqref{eq:shatbrac} it follows that under the substitution $u_m=q^m$ the latter equals a meromorphic Jacobi form. 
\begin{prp}\label{prp:ubrackst}
We have $\langle t_{(N)}(\zeta,\cdot)\rangle_{\vec{u}} = A_N(\zeta,\vec{u})$ and $\langle s(\varrho,\cdot)\rangle_{\vec{u}} = B(\varrho,\vec{u})$.
\end{prp}
\begin{proof}
 Note that $\langle 1 \rangle_{\vec{u}}=1$. By \Cref{lem:314} we directly find that
\begin{align*}
	\langle t_{(N)}(\zeta;\lambda) \rangle_{\vec{u}} &= \langle 1\rangle_{\vec{u}} \+ \sum_{m\ge1} (\zeta^m+\zeta^{-m})\,\langle \delta_{r_m(\lambda) \geq Nm} \rangle_{\vec{u}} \\
	&= 1 \+ \sum_{m\ge1}\sum_{r\geq 1} (\zeta^m+\zeta^{-m})\, \delta_{r = Nm} u_m^r = \sum_{m\in \mathbb{Z}} \zeta^m u_m^{Nm} = A_N(\zeta,\vec{u}), \\
	\langle s(\varrho;\lambda)\rangle_{\vec{u}} &= \left\langle\frac{\varrho}{1-\varrho}+\frac{1}{2}\right\rangle_{\vec{u}}\+ \sum_{m\ge1} \left(  \varrho^m \meno \varrho^{-m} \right)\langle r_m(\lambda) \rangle_{\vec{u}} \\
	&= \frac{\varrho}{1-\varrho}+\frac{1}{2}\+ \sum_{m\ge1}\sum_{r\geq 1} \left(  \varrho^m \meno \varrho^{-m} \right) u_m^r = \Bt(\varrho,\vec{u}). \qedhere
\end{align*}
\end{proof}

A reformulation of~\cite[Theorem 3.5.4]{vISymmetric} generalizing \Cref{lem:314} is the following. For $g,h:\N\to \Q$, define the \emph{discrete convolution product} of~$g$ and~$h$ by
\[(g*h)(n)\defis\sum_{j=1}^{n-1} g(j)h(n-j)\]
and denote the convolution product of functions~$g_1,\ldots, g_n$ by
\begin{equation*}\label{eq:conv}
	\bigconv_{j=1}^n g_j \defis g_1 * \cdots * g_n.
\end{equation*}
%
\begin{prp} \label{lem:connected}
	Let $n\in\N$ and $g_1,\ldots,g_n:\N_{0}\to \C$. For $m_1,\ldots,m_n\in\N$, write $G_j:\partitions\to \C$ for the function $G_j(\lambda) := g_j(r_{m_j}(\lambda))$. Then the following hold:
	\begin{enumerate}[leftmargin=*,label=\rm(\arabic*)]
\item \label{lem:connected 1} We have
	\[  
		\langle G_1  \otimes \cdots \otimes G_n \rangle_{\vec{u}} \=
		\delta_{m_1=\ldots=m_n}
		\sum_{r\ge1} g_{[n]}(r)u_m^r ,
	\]
	where $m:=m_1=\ldots=m_n$ and
	\[
		g_{[n]}\=\sum_{\alpha\in\Pi(n)} \mob \bigconv_{A\in \alpha} \partial g_A.
	\]
\item If, moreover, $g_n(r)=r$ for all $r\in \N_0$, then
	\[  
		\langle G_1  \otimes \cdots \otimes G_n \rangle_{\vec{u}} \=
		\delta_{m_1=\ldots=m_n}
		\sum_{r\ge1} r{g}_{[n-1]}(r)u_m^r.
	\]
\end{enumerate}
\end{prp}

We let $\mathcal{A}^\circ$ be the algebra generated by the $A_N$ for all $N\in \N$ and $\mathcal{B}^\circ$ be the algebra generated by the $\Bt$. 
\begin{thm}
The algebras $\mathcal{A}$ and $\mathcal{A}^\circ$ are isomorphic and $\mathcal{B}^\circ$ is contained in $\langle \mathcal{B} \rangle_{\vec{u}}$.
\end{thm}
\begin{proof}
	First we show that the image $\As$ under the $\vec{u}$-bracket is contained in $\mathcal{A}^\circ$. Let $\vec{N}=(N_1,\ldots,N_\ps)\in \N^\ps$. As every $\vec{u}$-bracket can expressed polynomially in terms of connected $\vec{u}$-brackets (see \eqref{eq:mobiusubrack}), it suffices to show that the connected bracket of the $t_{N_j}$, defined by~\eqref{eq:fa}, is an element of $\As^\circ$. By \Cref{lem:connected} (1), we obtain
	\[ \langle t_{N_1} \otimes t_{N_2} \otimes \cdots \otimes t_{N_\ps} \rangle_{\vec{u}} \= \delta_{\ps,1}+\sum_{m,r\geq 1}g_{m,\vec{N}}(r)u_m^r \prod_{j}\left(\zeta_j^m+\zeta_j^{-m}\right) ,\]
	where
	\begin{align*}
	g_{m,\vec{N}}(r) 
	&:=\sum_{\substack{\alpha\in\Pi(\vec{N}) \\ \alpha=\{A_1,\ldots,A_t\}}} \hspace{-0.5cm} \mob \!\!\!\sum_{\substack{\ell_1,\ldots,\ell_t\\ \ell_1+\ldots+\ell_t=r}} \prod_{j=1}^t \partial\left(x\mapsto \prod_{N\in A_j} \delta_{x \geq  Nm}\right)(\ell_j) \\
	&=\sum_{\substack{\alpha\in\Pi(\vec{N}) \\ \alpha=\{A_1,\ldots,A_t\}}} \hspace{-0.5cm} \mob \sum_{\substack{\ell_1,\ldots,\ell_t\\ \ell_1+\ldots+\ell_t=r}} \prod_{j=1}^t \delta_{\ell_j = (\max A_j)m} 
	= \sum_{\alpha\in\Pi(\vec{N})} \mob  \delta_{r=Mm}.
	\end{align*}
	Here we set $M:=\sum_{j=1}^t\max A_j$. Note that 
	$\sum_{\alpha \in \Pi(n)} \mob =0$ for $n\geq 2$. 
	Hence, the connected $\vec{u}$-bracket is a linear combination of elements of $\As$, as desired:
	\begin{equation}\label{eq:ubrackoffa}
	\langle t_{N_1} \otimes t_{N_2} \otimes \cdots \otimes t_{N_\ps} \rangle_{\vec{u}} 
	\= \frac{1}{2}\sum_{\alpha\in\Pi(\vec{N})} \mob \sum_{m\in \Z} u_m^{Mm} \prod_{j}\left(\zeta_j^m+\zeta_j^{-m}\right).
	\end{equation}

Next, we show that the $\vec{u}$-bracket restricts to an isomorphism between $\As$ and $\As^\circ$. Injectivity is clear since the $\vec{u}$-bracket is an isomorphism between algebras containing~$\As$ and~$\As^\circ$. For surjectivity, it suffices to show that the $\vec{u}$-bracket is surjective on monomials $\prod_{j=1}^d B(\varrho_j,\vec{u}) = \prod_{j=1}^d \langle t(\varrho_j,\cdot)\rangle_{\vec{u}}$ (by \Cref{prp:ubrackst}). We show the stronger statement that for $f,g\in \As$ there exists $h\in \As$ such that 
\begin{equation}\label{eq:claim}
	\langle f\rangle_{\vec{u}}\langle g\rangle_{\vec{u}} = \langle h \rangle_{\vec{u}}.
\end{equation} 
First of all, if $f=s(\varrho_1,\cdot)$ and $g=s(\varrho_2,\cdot)$, then 
\[
\langle f\rangle_{\vec{u}}\langle g\rangle_{\vec{u}} = \langle fg\rangle_{\vec{u}} - \langle f \otimes g\rangle_{\vec{u}}
\]
and by \eqref{eq:ubrackoffa} we have that $\langle f\otimes g\rangle_{\vec{u}}=\langle h\rangle_{\vec{u}}$ for some $h$. Finally, by \cite[Proposition~3.2.7]{vISymmetric} one can reduce the general case to this case. Hence, \eqref{eq:claim} follows, which implies the surjectivity.

	By Corollary~3.5.6 in~\cite{vISymmetric} the $n$-fold connected $\vec{u}$-bracket of $s$ equals
	\begin{equation}\label{eq:sss}\langle s \otimes \cdots\otimes s\rangle_{\vec{u}} \= \left(\frac{\varrho}{1-\varrho}+\frac{1}{2}\right)\delta_{n,1}+ \sum_{m,r\geq 1} \prod_{j=1}^n \left(\varrho_j^m-\varrho_j^{-m}\right) r^{n-1} u_m^r.
\end{equation}
Note that this is an element of $\mathcal{B}^\circ$ if $n=1$. The fact that $\mathcal{B}^\circ$ is contained in $\langle \mathcal{B} \rangle_{\vec{u}}$ follows analogously to the surjectivity above. 
\end{proof}

We are now ready to prove our first two main theorems.
\begin{proof}[Proof of \Cref{thm:1}] 
	As the $q$-bracket of elements in $\mathcal{A}$ are obtained by letting $u_j=q^j$ in~\eqref{eq:ubrackoffa}, both statements follow immediately.
\end{proof}
\begin{proof}[Proof of \Cref{thm:2}] 
Recall the Jacobi form $F(z,w;\tau)$; see~\eqref{eq:F}. Note this function has a pole if $z=0$ or $w=0$. The constant term in the expansion around $w=0$ equals 
\[ 
	\coef{w^{0}} F(z,w;\tau) = \frac{\z}{\z-1}-\frac{1}{2}-\sum_{m,r\geq 1} \left(\z^m-\zeta^{-m}\right) q^{mr} 
\]
and for $n\in\N$, the coefficient of $w^{n-1}$ of $F(z,w;\t)$ in the expansion around $w=0$ equals 
\[ 
	\frac{(n-1)!}{(2\pi i)^{n-1}}\coef{w^{n-1}} F(z,w;\tau) = \frac{\zeta_{1}}{\zeta-1}\delta_{n=1}+\frac{B_{n}}{n}-\sum_{m,r\geq 1} \left(\z^m+(-1)^n\z^{-m}\right)r^{n-1} q^{mr}.
\]
Writing $\varrho_j=e^{2\pi i \mathfrak{z}_j}$, we conclude that $\langle s\otimes \cdots \otimes s \rangle_{q}$, given by \eqref{eq:sss}, equals
\[ 
	\langle s \otimes \cdots\otimes s\rangle_{q} \= -\frac{(n-1)!}{2(2\pi i)^{n-1}}\coef{w^{n-1}}  \sum_{\bm\nu\in \{\pm 1\}^n}\!\! F(\nu_1\mathfrak{z}_1+\ldots+\nu_n \mathfrak{z}_n,w;\tau) \prod_{j}\nu_j.
\]
Since the coefficients of a Jacobi form are quasi-Jacobi forms by \cite[Subsubsection~1.3.3]{OP19} and \cite[Proposition~2.24]{vITaylor} (of which \Cref{prop:der} is a special case), we conclude $\langle s \otimes \cdots\otimes s\rangle_{q}$ is a quasi-Jacobi form. By~\eqref{eq:ubrackoffa} and \Cref{lem:connected}\,(2), for $n\in\N$ we have
	\begin{multline*}
		\langle t_{N_1} \otimes \cdots \otimes t_{N_\ps} \otimes s \otimes \cdots\otimes s\rangle_q\\
	=\frac{1}{2}\sum_{\alpha \in \Pi(\vec{N} )} \mob \sum_{m\geq 1}\prod_{j=1}^\ps \left(\zeta_j^m+\zeta_j^{-m}\right)\prod_{j=1}^n (\varrho_j^m-\varrho_j^{-m})
	(Mm)^n q^{Mm^2},
	\end{multline*}
where we take $n$ times the function~$s$ and $M$ is as in the proof of \Cref{thm:1}. 
\end{proof}
\begin{rmk}\itshape
We have
\[ \pdv{}{\mathfrak{z}_1}\cdots \pdv{}{\mathfrak{z}_n}\langle s \otimes \cdots\otimes s\rangle_{q} \=  \frac{1}{2}\left(q\pdv{}{q}\right)^{n-1}\!\!\sum_{\bm\nu\in \{\pm 1\}^n}\!\! E_2(\nu_1\mathfrak{z}_1+\ldots+\nu_n \mathfrak{z}n;\tau).
\]
Namely, by~\eqref{eq:sgenser}, we find that $\pdv{}{\mathfrak{z}_1}\cdots \pdv{}{\mathfrak{z}_n}\langle s \otimes \cdots\otimes s\rangle_{q}$  equals $2^n\cdot2\pi i$ times the left-hand side of \cite[Corollary 3.3.2]{vISymmetric} after substituting $z_j = 2\pi i \mathfrak{z}_j$. Similarly, we find $E_2(z;\tau) = (2\pi i)^2 P(2\pi i z)$, where the propagator $P$ was defined in \cite{vISymmetric}. The inversion formula for connected brackets \eqref{eq:connectedbracket} allows to replace the sum over $\alpha\in \Pi(n)$ by $A=\{1,\ldots,n\}$. Hence, also the right-hand sides agree up to the same factor and substitution. 
\end{rmk}

\section{Proof of \Cref{thm:falsemock}}\label{sec:falsemock}
In this section, we prove \Cref{thm:falsemock}, i.e., we show that $\langle t_N \otimes s^* \rangle_{q} $ is a linear combination of a Jacobi theta function, a false theta function, and an Appell--Lerch function. Set
\begin{align} \nonumber
\Theta^{\ast}(\z;q) &:= \Theta(z;\tau)=\sum_{n\in \Z} \zeta^n q^{n^2}, \quad T^{\ast}(\z;q) := T(z;\t) := \sum_{n\in \Z} \sgn\left(n+\frac{1}{2}\right) \zeta^n q^{n^2}, \\
f_N(z,w;\t) &:= \frac12 \sum_{\substack{m,n\in\Z}}  \left(\sgn\left(n+\frac12\right)+\sgn\left(m-\frac12\right)\right) e^{2\pi i(mz+nw)} q^{Nm^2+mn}.\label{E:fN}
\end{align}
Then, in~\eqref{tsbracket} we express $\langle t_N \otimes s^* \rangle_{q} $ as a certain linear combination of these functions. Moreover, in the next subsections, we obtain the transformation of $T$ as a false theta function and $f_N$ as an Appell--Lerch sum. 

\subsection{The false theta function $T$}
Let $\omega\in\H$. Define
\[
\widehat{T}(z;\tau,\omega) = \sum_{n\in\Z} E\!\left(-i\sqrt{2\pi i(\omega-\tau)}\left(n+\frac{z_2}{2\tau_2}\right)\right)  \zeta^{n} q^{n^2}.
\]
Then, for $\varepsilon>0$ and $0\leq z_2 \leq 2\tau_2$ we have
\begin{equation*}
	T(z;\tau) = \lim_{t\to\infty} \widehat T(z;\tau,\tau+it+\varepsilon).
\end{equation*}
The completed function~$\widehat{T}$ satisfies the following transformations. 

\begin{prp}\label{prop:Subgroup4N}\hspace{0cm}
	\begin{enumerate}[label=\rm(\arabic*),wide,labelwidth=!,labelindent=0pt]
		\item For $m,\ell \in \N$ one has
		\begin{equation*}
			\widehat{T}\!\left( z+2m\tau+\ell; \tau, \omega\right) = \zeta^{-m} q^{-\frac{m^{2}}{2}} \widehat{T}(z;\tau,\omega).
		\end{equation*}
		\item\label{I:gamma} For $\gamma=\abcd\in \Gamma_0(4)$ one has
		\begin{equation*}
			\hspace{-0.15cm}\widehat{T}\!\left(\frac{z}{c\tau+d}; \frac{a\tau+b}{c\tau+d},\frac{a\omega+b}{c\omega+d}\right) 
			\hspace{-0.1cm}=  \left(\frac{-1}{d}\right)\chi_{\tau,\omega}(\gamma) j(\gamma,\tau) e^{\frac{\pi icz^2}{2(c\tau+d)}}
			\widehat{T}\left(z;\tau,\omega\right),
		\end{equation*}
	\end{enumerate}
where the automorphy factor $j(\gamma,\tau)$ for $k=\frac{1}{2}$ is given by \eqref{eq:aut}.
\end{prp}
\begin{proof}
	\begin{enumerate}[label=\rm(\arabic*),wide,labelwidth=!,labelindent=0pt]
		\item The claim follows directly from \eqref{el}.
		\item Note that
		\begin{align*}
			T(z;\t) &= \sum_{n\in \Z+\frac{1}{2}} \sgn(n) \z^{n-\frac{1}{2}} q^{\left(n-\frac12\right)^2} = \zeta^{-\frac{1}{2}} q^{\frac{1}{4}} \sum_{n\in \Z+\frac{1}{2}} \sgn(n) e^{2\pi i n\left(z - \tau - \frac{1}{2} + \frac{1}{2}\right)} q^{n^{2}}\\
			&=\zeta^{-\frac{1}{2}} q^{\frac{1}{4}} \psi\left(z- \tau-\frac{1}{2};2\tau\right),
		\end{align*}
		where $\psi$ is the false theta function given in \eqref{eq:false}. Note that,
		\[
			\widehat{T}(z;\tau,\omega) = q^{\frac14} \z^{-\frac12} \widehat{\psi}\left(z-\t-\frac12;2\t,2\omega\right).
		\]
		Hence, for $\gamma=\abcd \in \Gamma_0(2)$, we obtain from \eqref{mod}
		\begin{align*}
			&\hspace{-1cm}\widehat{T}\left(\frac{z}{c\tau\!+\!d}; \frac{a\tau\!+\!b}{c\tau\!+\!d},\frac{a\omega\!+\!b}{c\omega\!+\!d}\right) 
			= e^{\frac{\pi i}{2} \frac{a\tau \!+\! b}{c \tau \!+\! d} \!-\!\frac{\pi i z}{c\tau \!+\! d}} \widehat{\psi}\left(\frac{z}{c\tau\!+\!d}\!-\!\frac{a\tau\!+\!b}{c\tau\!+\!d}\!-\!\frac12;\frac{a(2\tau)\!+\!2b}{\frac{c}{2}(2\tau)\!+\!d},\frac{a(2\omega)\!+\!2b}{\frac{c}{2}(2\omega)\!+\!d}\right) \\
			&= e^{\frac{\pi i}{2} \frac{a\tau + b}{c \tau +d} -\frac{\pi i z}{c\tau + d}}\chi_{2\tau,2\omega}
			\begin{psmallmatrix}
			 	a & 2b\\
			 	\frac{c}{2} & d
			\end{psmallmatrix}
			\nu_\eta
			\begin{psmallmatrix}
				a & 2b\\
				\frac{c}{2} & d
			\end{psmallmatrix}
			^3 (c\tau+d)^{\frac{1}{2}}\\
			&\hspace{1cm} \times e^{\frac{\pi i c}{2(c\tau + d)}\left(z - (a\tau + b) - \frac{1}{2}(c\tau + d)\right)^{{2}}} \widehat{\psi}\left(z-(a\tau+b)-\frac{1}{2}(c\tau+d);2\tau,2\omega\right).
		\end{align*}
		We now write
		\begin{equation*}
			\widehat{\psi}\left(z - (a\tau\!+\!b)-\frac{1}{2}(c\tau\!+\!d);2\tau,2\omega\right) = \widehat{\psi}\left(z-\tau-\frac{1}{2} + \left(\frac{1\!-\!a}{2}\!-\!\frac{c}{4}\right)2\tau+\frac{1\!-\!d}{2}-b; 2\tau,2\omega\right)\!.
		\end{equation*}
		Now assume that $4\mid c$. Then we use \eqref{el} with $\tau \mapsto 2\tau, \omega \mapsto 2\omega, m=\frac{1-a}{2} - \frac{c}{4}, \ell = \frac{1-d}{2}-b$, and $z\mapsto z-\tau-\frac{1}{2}$ to obtain
		\begin{multline*}
			\widehat{\psi}\left(z-\tau-\frac{1}{2}+\left(\frac{1-a}{2} - \frac{c}{4}\right) 2\tau + \frac{1-d}{2}-b;2\tau,2\omega\right)\\
			= (-1)^{\frac{1-a}{2} - \frac{c}{4} + \frac{1-d}{2} + b} e^{2\pi i \left(\frac{a-1}{2} + \frac{c}{4}\right)\left(z-\tau-\frac{1}{2}\right) - 2N\pi i \left(\frac{1-a}{2}-\frac{c}{4}\right)^{2} \tau} \widehat{\psi}\left(z-\tau-\frac{1}{2};2\tau,2\omega\right).
		\end{multline*}
		Now, a computations using $ad-bc=1$ and $4\mid c$ yields
		\begin{multline*}
			e^{\frac{\pi i}2\frac{a\tau+b}{c\tau+d}-\frac{\pi iz}{c\tau+d}+\frac{\pi ic}{2(c\tau+d)}\left(z-(a\tau+b)-\frac12(c\tau+d)\right)^2} e^{2\pi i\left(\frac{a-1}2+\frac c{4}\right)(z-\tau)-2\pi i\left(\frac{1-a}2-\frac c{4}\right)^2\tau}\\
			= e^{\frac{\pi i\tau}2-\pi iz+\frac{\pi icz^2}{2(c\tau+d)}+2\pi i\left(\frac{ab}{4} + \frac{cd}{16}\right)}.
		\end{multline*}
		Also, note
		\[\chi_{2\tau,2\omega}
			\begin{pmatrix}
			 	a & 2b\\
			 	\frac{c}{2} & d
			\end{pmatrix} = \chi_{\tau,\omega}\begin{pmatrix}
			 	a & b\\
			 	c & d
			\end{pmatrix}. \]
		We thus have
		\begin{align*}
			&\widehat{T}\!\left(\frac{z}{c\tau+d}; \frac{a\tau+b}{c\tau+d},\frac{a\omega+b}{c\omega+d}\right)
			=(-1)^{\frac{1-d}{2} + b}  \chi_{\tau,\omega}
			\begin{pmatrix}
			 	a & b\\
			 	c & d
			\end{pmatrix}
			\nu_\eta
			\begin{pmatrix}
				a & 2b\\
				\frac{c}{2} & d
			\end{pmatrix}
			^3 (c\tau+d)^{\frac{1}{2}}\\
			&\hspace{5cm}\times e^{\frac{\pi i\tau}2-\pi iz+\frac{\pi icz^2}{2(c\tau+d)}+2\pi i\left(\frac{ab}{4} + \frac{cd}{16}\right)}
			\widehat{\psi}\left(z-\tau-\frac{1}{2};2\tau,2\omega\right) \\
			&\quad= (-1)^{\frac{1-d}{2} + Nb}  \chi_{\tau,\omega}
			\begin{pmatrix}
			 	a & b\\
			 	c & d
			\end{pmatrix}
			\nu_\eta
			\begin{pmatrix}
				a & 2b\\
				\frac{c}{2} & d
			\end{pmatrix}
			^3 (c\tau+d)^{\frac{1}{2}} e^{\frac{\pi icz^2}{2(c\tau+d)}+2\pi i\left(\frac{ab}{4} + \frac{cd}{16}\right)}
			\widehat{T}\left(z;\tau,\omega\right) .
		\end{align*}
		Using \eqref{E:MultEt}, it is not hard to finish the proof of \ref{I:gamma}.\qedhere
	\end{enumerate}
\end{proof}

\subsection{The Appell--Lerch function $f_N$}
Next, we want to find modularity properties of $f_N$ defined in \eqref{E:fN}. In particular, we are interested in the cases $z\in\{0,Nw\}$.  
Define
\begin{multline*}
	\widehat f_N(z,w;\t) = f_N(z,w;\t)\\
	+  \frac{i}{2}\zeta^{-N} \sum_{k=0}^{2N-1} \xi^k \vartheta\!\left(z\!-\!N\t\!+\!k\t\!+\!N\!-\!\frac{1}2;2N\t\right) R\!\left(2N w\!-\!z\!+\!N\t\!-\!k\t\!-\!N\!+\!\frac{1}{2};2N\t\right)\!.
\end{multline*}
The completed function~$\widehat f_N$ satisfies the following transformations.
\begin{prp} 
	For $m_1, m_2, \ell_1, \ell_2\in\Z$, we have
	\begin{align*}
		\widehat f_N(z+m_1\t+\ell_1,w+m_2\t+\ell_2;\t)
		=  \z^{-m_2}\xi^{2N m_2-m_1} q^{N m_2^2 -m_1m_2} \widehat f_{N}(z,w;\t).
	\end{align*}
	For $\abcd\in \sltwoz$, we have
	\[
		\widehat f_N \left(\frac {z}{c\t+d},\frac {w}{c\t+d};\frac{a\t+b}{c\t+d}\right) = (c\t+d) e^{\frac{2 \pi i c}{c\t+d}\left(-Nw^2+zw\right)} \widehat f_{N}(z,w;\t).
	\]
\end{prp}
\begin{proof}
	We write 
	\begin{equation*}
		f_N(z,w;\tau) = \left(\sum_{\substack{n\ge0\\m\ge1}} - \sum_{\substack{n\leq -1\\m\leq 0}}\right) e^{2\pi i(mz+nw)} q^{Nm^2+mn}.
	\end{equation*}
	Assume that $1<|w|<|q|^{-1}$, i.e., $0<-w_2<\tau_2$. 
	For $m\ge1$, we have 
	\begin{equation*}
		\sum_{n\ge0} e^{2\pi i nw} q^{mn} = \frac1{1-e^{2\pi i w}q^m}.
	\end{equation*}
	For $m\le 0$, we have
	\begin{equation*}
		\sum_{n\le-1} e^{2\pi i n w} q^{mn} = \sum_{n\ge1} e^{-2\pi i n w}q^{-mn} = \frac{e^{-2\pi i w}q^{-m}}{1-e^{-2\pi i w}q^{-m}} = - \frac1{1-e^{2\pi i w}q^m}.
	\end{equation*}
	Thus we have
	\begin{equation*}
		f_N(z,w;\t) = \sum_{m\in\Z} \frac{e^{2\pi i mz}q^{Nm^2}}{1-e^{2\pi i w}q^m}.
	\end{equation*}
	We write this in terms of Zwegers' multivariable Appell functions (see \eqref{eq:appel}) as 
	\begin{align*}
		f_N(z,w;\t) &=   \sum_{m\in\Z} \frac{q^\frac{2Nm(m+1)}2 q^{-Nm}\z^{m} }{1-\xi q^m}  = \xi^{-N} A_{2N}(w,z-N\t;\t),
	\end{align*}
	from which the completion of $f_N$ follows directly, i.e.,
	\begin{equation*}
		\widehat{f}_N(z,w;\tau) = \xi^{-N} \widehat{A}_{2N}(w,z-N\tau;\tau).
	\end{equation*}
Using \eqref{E:ellA}, we obtain $\widehat A_{2N}(w,z-N\tau;\tau) = \xi^N \widehat A_{2N}(w,z;\tau)$ and thus
\begin{equation*}
	\widehat f_N(z,w;\tau) = \widehat A_{2N}(w,z;\tau),
\end{equation*}
so that the transformations follow from \eqref{E:transA} and \eqref{E:ellA}.
\end{proof}
The functions $  \widehat f_{N}(0,z;\t)$ and $\widehat f_{N}(Nz,z;\t)$ transform as Jacobi forms of weight~$1$ and index~$-N$ and~$0$, respectively:
\begin{cor}\label{cor:fN}
	For $m, \ell\in\Z$, we have
	\begin{align*}
		\widehat f_N(0,z+m\t+\ell;\t)	&= \z^{2N m} q^{N m^2} \widehat f_{N}(0,z;\t), \\
		\widehat f_N(N(z+m\t+\ell),z+m\t+\ell;\t) &=   \widehat f_{N}(Nz,z;\t).
	\end{align*}
For $\abcd\in \sltwoz$, we have
\begin{align*}
	\widehat f_N \left(0,\frac{z}{c\t+d};\frac{a\t+b}{c\t+d}\right) &= (c\t+d) e^{-\frac{2\pi icNz^2}{c\t+d}} \widehat f_{N}(0,z;\t), \\
	\widehat f_N \left(\frac {Nz}{c\t+d},\frac{z}{c\t+d};\frac{a\t+b}{c\t+d}\right) &= (c\t+d) \widehat f_{N}(Nz,z;\t).
\end{align*}
\end{cor}

\subsection{Proof of \Cref{thm:falsemock}}
\begin{proof}[{Proof of \Cref{thm:falsemock}}] Let $N\in \Z$. Consider functions on $\N$ given by 
\[ D_{N,m}(r)=\delta_{r\geq Nm}, \quad  X_\xi(r) = \sum_{j=1}^r\xi^j.\]
Now, first we compute using \Cref{lem:connected}
	\[
		\langle t_N \otimes s^*\, \rangle_{q} \= \sum_{m,r\geq 1} \left(\zeta^m+\zeta^{-m}\right)\left(\varrho^m g_{N,m,\xi}(r)\meno\varrho^{-m}g_{N,m,\xi^{-1}}(r)\right)  q^{mr}
	\]
	with
	\begin{equation*}
	 g_{N,m,\xi}(r) \= \partial\left(D_{N,m}X_\xi\right) \meno \partial(D_N) \ast \partial\left(X_\xi\right) \= \delta_{r,Nm}X_\xi(r) \+ \xi^r\delta_{r >Nm} \meno \xi^{r-Nm}\delta_{r>Nm}.
	\end{equation*}
	Hence, we have
	\begin{align*}
	\langle t_N \otimes s^*\, \rangle_{q} 
	= \sum_{m\geq 1} \left(\zeta^m+\zeta^{-m}\right)\varrho^m\left(\left(\xi+\ldots+\xi^{Nm}\right) q^{Nm^2} \+ \sum_{r > Nm}
	\left(\xi^r-\xi^{r-Nm}\right) q^{mr}\right) \hspace{1.4cm}\\
	-  \sum_{m\geq 1} \left(\zeta^m+\zeta^{-m}\right)\!\varrho^{-m}\!\left(\left(\xi^{-1}+\ldots+\xi^{-Nm}\right)\! q^{Nm^2} \+ \sum_{r > Nm}
	\left(\xi^{-r}-\xi^{-r+Nm}\right)\! q^{mr}\right) \\
	= \sum_{m\in \Z} \sgn(m)\! \left(\zeta^m+\zeta^{-m}\right)\!\varrho^m\! \left(\left(\xi^{\sgn(m)}+\ldots+\xi^{Nm}\right)\! q^{Nm^2}
	+ \sum_{\substack{s\in \Z \\ ms>0}}\left(\xi^{s+am}-\xi^s\right)\! q^{m(Nm+s)}\right)\!.
	\end{align*}
Note that for $m\neq 0$ one has
\[ 
	\xi^{\sgn(m)} = \frac{\xi^{m}-\xi^{-m}}{2} + \sgn(m)\frac{\xi^{m}+\xi^{-m}}{2}. 
\]
Hence,
\begin{multline*}
\sum_{m\in \Z} \sgn(m) \left(\zeta^m+\zeta^{-m}\right)\varrho^m \left(\xi^{\sgn(m)}+\ldots+\xi^{Nm}\right) q^{Nm^2} \\
=\sum_{m\in \Z} \left(\zeta^m+\zeta^{-m}\right)\varrho^m \frac{1-\xi^{-Nm}}{2 (\xi-1)}\rb{\left(\xi^{Nm+1}+1\right)+\sgn\left(m+\frac{1}{2}\right)\left(\xi^{Nm+1}-1\right)} q^{Nm^2}.
\end{multline*}
We conclude
\begin{align}\label{tsbracket}
&\langle t_N \otimes s \rangle_{q}  \nonumber
=\frac{1}{2(\xi-1)}\Bigl(\Theta^{\ast}\!\left(\z\varrho\xi^N;q\right)+\Theta^{\ast}\!\left(\z^{-1}\varrho\xi^N;q\right)-\Theta^{\ast}\!\left(\z\varrho\xi^{-N};q\right)-\Theta^{\ast}\!\left(\z^{-1}\varrho\xi^{-N};q\right)\\*
&\hspace{0.6cm} + T^{\ast}\!\left(\z\varrho\xi^N;q^N\right)+T^{\ast}\!\left(\z^{-1}\varrho\xi^N;q^N\right)+T^{\ast}\!\left(\z\varrho\xi^{-N};q^N\right)+T^{\ast}\!\left(\z^{-1}\varrho\xi^{-N};q^N\right)\\*
&\hspace{0.6cm}-2T^{\ast}\!\left(\zeta\varrho;q^N\right)-2T^{\ast}\!\left(\zeta^{-1}\varrho;q^N\right)\Bigr) + f_N\left(Nz,z;\tau\right)-f_N\left(0,z;\tau\right) \nonumber.
\end{align}
The transformation properties of $T$ and $f_N$ in \Cref{prop:Subgroup4N} and \Cref{cor:fN} complete the proof.
\end{proof}

\end{document}